\documentclass[11pt, leqno]{scrartcl}

\usepackage{amssymb,amsmath,amsthm}
\usepackage{latexsym}
\usepackage{t1enc}
\usepackage{lmodern}
\usepackage[utf8]{inputenc}
\usepackage[dvipsnames]{xcolor}
\usepackage{todonotes}
\usepackage{paralist}
\usepackage{enumitem}
\usepackage{esint}
\usepackage{tikz}
\usepackage{pgfplots}
\usepackage{xcolor}

\newtheorem{theorem}{Theorem}

\newtheorem{corollary}{Corollary}
\newtheorem{remark}{Remark}
\newtheorem{lemma}{Lemma}

\usepackage[pdfproducer={LaTeX with hyperref package},pdfpagelayout=SinglePage,bookmarksnumbered,pdftex=true,breaklinks=true,bookmarks=true,linktocpage=true,pdfpagelabels=true,plainpages=false,pagebackref=false]{hyperref} 
\hypersetup{colorlinks=true,linkcolor=blue,citecolor=blue,urlcolor=blue,filecolor=blue}
    \newcommand{\bilap}[1]{\Delta^{\!2}\hspace{-0.25mm}#1 }
    \newcommand{\grad}[1]{\nabla \! #1}
    \newcommand{\lap}[1]{\Delta \hspace{-0.15mm} #1 }
    \newcommand{\aq}{\Leftrightarrow} 
    \renewcommand{\epsilon}{\varepsilon}
    \newcommand{\eps}{\varepsilon}
    \newcommand{\norm}[1]{\|#1\|} 
    \newcommand{\abs}[1]{|#1|}
     \newcommand{\Ie}{\mathcal{I}_\epsilon} 
     \newcommand{\Iei}[1]{\mathcal{I}^{(#1)}_{\epsilon}} 
    \renewcommand{\L}{\Lambda}  
    \newcommand{\Le}{\L_\epsilon^{(k)}}   
     \newcommand{\Lei}[1]{\L^{(#1)}_\eps}   
     \newcommand{\R}{\mathbb{R}}   
    \newcommand{\N}{\mathbb{N}}    
    \newcommand{\uei}[1]{u^{(#1)}_{\epsilon}}	
    \newcommand{\ue}{u^{(k)}_{\epsilon}}	   
    \newcommand{\pe}[1]{p^{(#1)}_{\epsilon}}	
   \newcommand{\Ray}{\mathcal{R}}

 \renewcommand{\O}{\mathcal{O}}
  \newcommand{\dist}{\operatorname{dist}}

\begin{document}

\title{Existence of an optimal domain for the buckling load of a clamped plate with prescribed volume.}
\date{}
\author{Kathrin Stollenwerk \footnote{stollenwerk@instmath.rwth-aachen.de}}
\pagestyle{myheadings}
\maketitle 


\noindent\textbf{Abstract } We formulate the minimization of the buckling load of a clamped plate as a free boundary value problem with a penalization term for the volume constraint. As the penalization parameter becomes small we show that the optimal shape problem with prescribed volume is solved. In addition, we discuss two different choices for the penalization term.

\bigskip 
\noindent\textbf{Key words:} Inequalities involving eigenvalues, buckling load, fourth order

\bigskip
\noindent\textbf{MSC2010}: 49K20, 49R05, 15A42


\section{Introduction}

We consider the following variational problem. For $n \geq 2$ let $\Omega \subset \R^n$ be an open bounded domain. Then we define 
\[
 \Ray(v,\Omega) := \frac{\int_\Omega \abs{\lap v}^2dx }{\int_\Omega \abs{\grad v}^2dx}
\]
for $v \in H^{2,2}_0(\Omega)$ and denote $\Ray(v,\Omega)=\infty$ if the denominator vanishes. The quantity 
\[
  \Lambda(\Omega) := \inf \{\Ray(v,\Omega) : v \in H^{2,2}_0(\Omega)\}
\]
is called the buckling load of $\Omega$. The infimum is attained by the first eigenfunction $u$, which solves the following Euler-Lagrange equation
\[
\begin{cases}
  \bilap u + \Lambda(\Omega)\lap u &=0\; \mbox{in } \Omega \\
  u = \abs{\grad u} &=0 \; \mbox{on } \partial\Omega
\end{cases}
\] 
if the boundary of $\Omega$ is smooth enough. 

In 1951, G.\,Polya and G.\,Szeg\"o conjectured that among all domains of given measure the ball minimizes the buckling load (see \cite{PolyaSzego}). 

Up to now, this conjecture is still open. However, some partial results are known. 
In \cite{Szego50} Szeg\"o 
proved the conjecture for all smooth plane domains under the additional assumption that 
$u > 0$ in $\Omega$. 
M.\,S.\,Ashbaugh and D.\,Bucur proved that among all simply connected plane domains of 
prescribed volume there exists an optimal domain \cite{BucurAshbaugh}.  
In 1995, H.\,Weinberger and B.\,Willms proved the following uniqueness result for $n=2$, see \cite{Willms95}. 
If an optimal simply connected bounded plane domain $\Omega$ exists and if $\partial\Omega$ is smooth 
(at least $C^{2,\alpha}$), then $\Omega$ is a disc.  
This result has been extended to arbitrary dimension in 2015, see \cite{StoWa2015}. 

In the present paper, we will prove the following main theorem. 

\begin{theorem}\label{theo:main}
  There exists a bounded domain $\Omega^\ast$ with $\abs{\Omega^\ast}=\omega_0$ such that 
\[ 
  \L(\Omega^\ast) = \min\{\L(D): D\subset B , D \mbox{ open}, \abs{D}\leq\omega_0\},
\]
where $\abs{D}$ denotes the $n$-dimensional Lebesgue measure of $D$, $\omega_0>0$ is a given quantity and $B \subset \R^n$ is a ball with $\abs{B}>\omega_0>0$.  
 \end{theorem}

In order to prove Theorem \ref{theo:main}, we introduce a penalized variational problem following an idea of  H.\,W.\,Alt and L.\,A.\,Caffarelli in \cite{AltCaf81}.  Let $B \subset \R^n$ be a  ball  and choose $0<\omega_0 \ll \abs{B}$, where $\abs{B}$ denotes the $n$-dimensional Lebesgue-measure of $B$. For $\eps>0$  we define the penalization term $\pe{0} : \R \to \R$ by 
\begin{equation}\label{eq:pe0}
  \pe{0}(s) := \begin{cases}
     \frac{1}{\eps}(s-\omega_0), & s\geq \omega_0 \\
     0, & s \leq \omega_0
   \end{cases}
\end{equation}
and $\Iei{0} : H^{2,2}_0(B) \to \R$ by 
\[
   \Iei{0}(v) := \Ray(v,B) + \pe{0}(\abs{\O(v)}), 
\]
where $\O(v) := \{x\in B: v(x)\neq 0 \}$. 
We will see that for each $\eps>0$ there exists a minimizer $\uei{0} \in H^{2,2}_0(B)$ for $\Iei{0}$ and that $\uei{0}$ yields a domain $\Omega(\uei{0})$ with $\abs{\Omega(\uei{0})}=\omega_0$ which minimizes the buckling load among all open subsets $B$ with measure smaller or equal than $\omega_0$. This will prove Theorem \ref{theo:main}.

In view of the conjecture of Polya and Szegö, the next reasonable step would be to analyze regularity properties of the free boundary $\partial\Omega(\uei{0})$. If we orientate ourselves on the pioneering work of Alt and Caffarelli in \cite{AltCaf81}, the next way points towards qualitative properties of the free boundary would be establishing the nondegeneracy of $\uei{0}$ and, subsequently, proving that the free boundary has got a positive Lebesgue-density in every point. However, looking at nondegeneracy results for second order problems as in \cite{AguAltCaf1986}, \cite{AltCaf81} or \cite{BaWa09}, e.g., we see that in these settings a nondegeneracy result for the minimizing function is achieved by constructing suitable testfunctions which heavily rely on properties of the Sobolev space $H^{1,2}$. In our $H^{2,2}$ setting we do not possess any comparison principle. It seems that this might be the end point for our approach via the functional $\Iei{0}$. 

Consequently, we revise the functional $\Iei{0}$. Following an idea of N.\,Aguilera, H.\,W.\,Alt and L.\,A.\,Caffarelli in \cite{AguAltCaf1986}, we replace the penalization term $\pe{0}$ by a penalization term $\pe{1}$ which rewards volumes less than $\omega_0$ with a negative contribution to the functional. For that purpose, we define $\pe{1}:\R\to\R$ by
\begin{equation}\label{eq:pe1}
  \pe{1}(s) := \begin{cases}
                    \frac{1}{\eps}(s-\omega_0), & s\geq \omega_0 \\
						\eps(s-\omega_0), &s\leq \omega_0
					\end{cases}
\end{equation}
 and $\Iei{1}: H^{2,2}_0(B) \to\R$ by
 \[
    \Iei{1}(v) := \Ray(v,B) + \pe{1}(\abs{\O(v)}).
 \] 
We will prove that for every $\eps>0$ there exists a minimizer $\uei{1} \in H^{2,2}_0(B)$ for $\Iei{1}$ and that $\uei{1}$ yields a domain $\Omega(\uei{1})$ which minimizes the buckling load among all open subsets of $B$ with the same measure as $\Omega(\uei{1})$. The analysis of the volume of $\Omega(\uei{1})$ will be more challenging than in the case of $\Omega(\uei{0})$ since the rewarding part of the penalization term counteracts the monotonicity of the buckling load with respect to set inclusion . We will give more details on this issue in the sequel.  

The present paper is organized as follows. 
In Section \ref{sec:penprob}, we prove the existence of a minimizer $\uei{k}$ for the functional $\Iei{k}$ $(k=0,1)$ and show that $\uei{k} \in C^{1,\alpha}(\overline{B})$ for every $\alpha \in (0,1)$. Thereby, we refine the approach presented in \cite{Sto2015}. That paper deals with the functional $\Iei{0}$, but considers only $n=2$ and $n=3$. 
Consequently, the $C^{1,\alpha}$ regularity of $\uei{0}$ follows with the help of Sobolev's Embedding Theorem. In the present paper, we consider arbitrary $n\geq 2$ and refine the approach to the $C^{1,\alpha}$ regularity from \cite{Sto2015} by applying a bootstrap argument.
In Section \ref{sec:volcond}, we analyze the volume of the domains $\Omega(\uei{0})$ and $\Omega(\uei{1})$ separately. In Section \ref{sec:nonrewarding}, we consider the functional $\Iei{0}$ and prove our main theorem, Theorem \ref{theo:main}, by scaling arguments.  In Section \ref{sec:rewarding}, we consider the functional $\Iei{1}$. Thus, we deal with a penalization term, which rewards volumes less  than $\omega_0$ with a negative contribution to the functional. Hence, the behavior of the penalization term antagonizes the monotonicity of the buckling load with respect to set inclusion and  we cannot adopt the approach from Section \ref{sec:nonrewarding} to show that the optimal domain's volume cannot become less than $\omega_0$.  In fact, we will use an inequality by M.\,S.\,Ashbaugh and R.\,S.\,Laugesen (see \cite{AshLau1996}) to balance the rewarding penalization term and the monotonicity of the buckling load. Hence, provided $\eps$ is sufficiently small, $\abs{\Omega(\uei{1})} \in [\alpha_0\omega_0,\omega_0]$, where $\alpha_0 \in (\frac{1}{2},1)$ depends on $n, \omega_0$ and $\eps$. 
After refining the choice of $\eps$, we will obtain the following dichotomy:  there either holds 
\begin{compactitem}\setlength{\itemsep}{0pt}
\item  $\abs{\Omega(\uei{1})}=\omega_0$ or \\ 
\item $\abs{\Omega(\uei{1})}<\omega_0$ and if $\Omega(\uei{1})$ is scaled to the volume $\omega_0$, this enlarged domain is neither a subset of $B$ nor can it be translated into $B$.
\end{compactitem}  
In the first case, $\uei{1}$ is a minimizer of the functional $\Iei{0}$ and we can treat both functionals as equivalent in the sense that they are minimized by the same functions. 
Under the additional assumption that the free boundary $\partial\Omega(\uei{1})$ satisfies a doubling property, we will disprove the occurrence of the latter case. We assume that there exists a constant $\sigma>0$ such that for every $x_0 \in \partial\Omega(\uei{1})$ and every $0<R\leq R_0$ there holds 
\[
   \abs{B_{2R}(x_0)\cap\Omega(\uei{1})} \leq \sigma \abs{B_R(x_0)\cap\Omega(\uei{1})}. 
\]
This assumed doubling property enables us to establish a nondegeneracy property of the minimizing function $\uei{1}$ and, subsequently, we can show that the latter case of the above mentioned dichotomy cannot occur. 
We should emphasize that for proving the nondegeneracy of $\uei{1}$, besides the assumption of the doubling property, the crucial incredient is the rewarding part of the penalization term $\pe{1}$.


\section{The penalized problems}\label{sec:penprob}

In this section, we analyze the penalized problems 
\[
  \min\{\Iei{k}(v); v \in H^{2,2}_0(B)\}
\]
for $k\in\{0,1\}$. Proving the existence of  a minimizer $\ue$ for every $\eps >0$ is our first step. For the proof we refer to \cite[Theorem 2.1]{Sto2015}. 
\begin{theorem}
 For every $\eps>0$ and for $k \in \{0,1\}$ there exists a function $\ue \in H^{2,2}_0(B)$ such that
\[
   \Iei{k}(\ue) = \min\{\Iei{k}(v): v\in H^{2,2}_0(B)\}. 
\]
\end{theorem}
Without loss of generality, we assume $\ue$ to be normalized in the sense that 
\[
   \int\limits_B \abs{\grad \ue}^2dx = 1
\] 
and denote
\[
   \Le := \Ray(\ue) = \int\limits_B \abs{\lap \ue}^2\,dx.
\]
Note that $\O(\ue) = \{\ue(x)\neq0\}$ is a non-empty set. By the absolute continuity of the Lebesgue integral, the $n$-dimensional Lebesgue measure of $\O(\ue)$ cannot vanish. 
Since the minimizer $\ue \in H^{2,2}_0(B)$, we do not know any continuity properties of $\uei{k}$ yet. In particular, we do not know if $\O(\ue)$ is an open set. To prove continuity of $\ue$ we will apply an idea of Q.\,Han and F.\,Lin in \cite{HanLin}, which is based on a bootstrap argument and Morrey's Dirichlet Growth Theorem. This ansatz will lead to the $C^{1,\alpha}$ regularity of $\ue$ for every $\alpha \in (0,1)$ in arbitrary dimension $n$. 

In the sequel, we will apply the following version of Morrey's Dirichlet Growth Theorem (see \cite[Theorem 3.5.2]{morrey}). 
\begin{theorem}[Morrey's Dirichlet Growth Theorem]\label{theo:morrey}
 Suppose $\phi \in H^{1,p}_0(B)$, $1 \leq p\leq n$, $0< \alpha\leq1$ and  suppose there exists a constant $M>0$ such that 
 \[
  \int\limits_{B_r(x_0)\cap B}\abs{\grad \phi}^pdx \leq M \, r^{n-p+\alpha p}
 \]
for every $B_r(x_0)$ with $x_0 \in \overline{B}$. Then $\phi \in C^{0,\alpha}(\overline{B})$.
\end{theorem}
In addition, we need a technical lemma cited from \cite[Chapter III, Lemma 2.1]{giaq_mult_int}.

\begin{lemma}\label{la:tech_morrey}
 Let $\Phi$ be a nonnegative and nondecreasing function on $[0,R]$. Suppose that there exist positive constants $\gamma, \alpha, \kappa, \beta$, $\beta<\alpha$, such that for all $0\leq r\leq R \leq R_0$ 
\[
  \Phi(r) \leq \gamma\left[ \left(\frac{r}{R}\right)^\alpha + \delta \right]\Phi(R)+\kappa R^\beta.
\]
Then there exist positive constants $\delta_0 =\delta_0(\gamma,\alpha,\beta)$ and $C=C(\gamma, \alpha,\beta)$ such that if $\delta < \delta_0$, for all $0\leq r\leq R\leq R_0$ we have
\[
  \Phi(r) \leq C \left(\frac{r}{R}\right)^\beta \left[\Phi(R) + \kappa R^\beta\right]  .
\] 
\end{lemma}

Let us fix $R_{\eps,k}$ as 
\[
  R_{\eps,k} := \min\left\{1, \left(\frac{\abs{\O(\ue)}}{2\omega_n}\right)^\frac{1}{n}\right\}.
\]
Now suppose $x_0 \in \overline{B}$ and choose $0<r\leq R\leq R_{\eps,k}$. We define $\hat{v}_k \in H^{2,2}_0(B)$ by 
\begin{equation}\label{eq:hat_v}
 \hat{v}_k = \begin{cases}
   \ue, &\mbox{in } B\setminus B_R(x_0) \\
   v_k, & \mbox{in } B_R(x_0)
 \end{cases},
\end{equation}
for a $v_k \in H^{2,2}(B_R(x_0)\cap B)$ such that $v_k-\ue \in H^{2,2}_0(B_R(x_0)\cap B)$ and $\bilap v_k =0$ in $B_R(x_0)\cap B$. 
We restrict ourselves to case $B_R(x_0)\cap  \O(\ue)\neq \emptyset$; otherwise, $\ue-v_k$ vanishes in $B_R(x_0)$.

As in \cite[Lemma 2.1]{Sto2015} we prove the next lemma. 
\begin{lemma}\label{biharm_v}
  Using the above notation, there exists a constant $C>0$ which only depends on $n$ such that for $0\leq r < R$ there holds
\begin{equation*}
    \int\limits_{B_r(x_0)\cap B}\abs{D^2 v_k}^2dx \leq C\,\left(\frac{r}{R}\right)^n \,\int\limits_{B_R(x_0)\cap B}\abs{D^2\ue}^2dx\,.
\end{equation*}
\end{lemma}

The next lemma will be the starting point for our bootstrap argument. 

\begin{lemma}\label{la:est2}
   Let $\ue$ be a minimizer of $\Iei{k}$ and $v_k-\ue \in H^{2,2}_0(B_R(x_0)\cap B)$. Then there exists a constant $C = C(n,\Le,\abs{\O(\ue)}) > 0$ such that for each $x_0 \in \overline{B}$ and each $0<R\leq R_{\eps,k}$ there holds
 \[
   \int\limits_{B_R(x_0)\cap B}\abs{D^2 (\ue-v_k)}^2dx \,\leq\,  C(n,\Le,\abs{\O(\ue)})\left(R^n+\int\limits_{B_R\cap B}\abs{\grad\ue}^2dx\right).
 \] 
\end{lemma}
\begin{proof}
 We obtain the result by comparing the $\Iei{k}$-energies of $\ue$ and $\hat{v}_k$, where  $\hat{v}_k$ is defined as in \eqref{eq:hat_v}.
Note that $\O(\ue) \not\subset B_R(x_0)$ due to the definition of $R_{\eps,k}$. However, we have to distinguish two different cases  depending on the volume of $\O(\hat{v}_k)$. First, let us consider that $\abs{\O(\hat{v}_k)} \leq \abs{\O(\ue)}$. 
Then the minimality of $\ue$ for $\Iei{k}$ implies
\[
 \int\limits_{B_R(x_0)\cap B}\abs{\lap\ue}^2-\abs{\lap{v}_k}^ 2dx \leq \Le \int\limits_{B_R(x_0)\cap B}\abs{\grad\ue}^2dx.
\]
Since $v_k$ is biharmonic in $B_R(x_0)\cap B$ and $v_k-\ue \in H^{2,2}_0(B_R(x_0)\cap B)$, we obtain
\[
 \int\limits_{B_R(x_0)\cap B}\abs{D^2(\ue-v_k)}^2dx \leq \Le \int\limits_{B_R(x_0)\cap B}\abs{\grad\ue}^2dx 
\]
and the claim is proven. 

Now let us assume that there holds 
\[ 
 \abs{\O(\ue)} < \abs{\O(\hat{v}_k)} \leq \abs{\O(\ue)}+\abs{B_R}.
\] 
In order to avoid the penalization term while comparing $\Iei{k}(\ue)$ and $\Ie(\hat{v}_k)$ we scale $\hat{v}_k$. 
For this purpose, we set 
\[
   \mu := \left( \frac{\abs{\O(\hat{v}_k)}}{\abs{\O(\ue)}}\right)^\frac{1}{n} >1. 
\] 
Without loss of generality, we think of $B$ as of a ball with center in the origin and radius $R_B$. Then  we denote $B^\ast := B_{\mu^{-1}{R_B}}$
and $w_k(x) := \hat{v}_k(\mu x)$ for $x \in B^\ast$.  Consequently, $w_k\in H^{2,2}_0(B^\ast) \subset H^{2,2}_0(B)$ and $\abs{\O(w_k)}=\abs{\O(\ue)}$. 
The minimality of $\ue$ for $\Iei{k}$ in $H^{2,2}_0(B)$ then implies
\begin{equation*}
\Le\int_B\abs{\grad w_k}^2 dy \leq \int\limits_B\abs{\lap w_k}^2dy \; \Leftrightarrow \; \Le\mu^{-2}\int\limits_{B}\abs{\grad \hat{v}_k}^2dx \leq \int\limits_{B}\abs{\lap\hat{v}_k}^2dx\,.
\end{equation*}
Rearranging terms we obtain the local inequality
\[
 \int\limits_{B_R\cap B}\abs{\lap{\ue}}^2- \abs{\lap{v}_k}^2dx \leq \Le\left(1-\frac{1}{\mu^2}\right) + \frac{\Le}{\mu^2}\left[\,\,\int\limits_{B_R\cap B}\abs{\grad{\ue}}^2-\abs{\grad{v_k}}^2dx\right],
\]
where we denote $B_R= B_R(x_0)$ for simplicity. Since $v_k$ is biharmonic in $B_R\cap B$ and $v_k-\ue \in  H^{2,2}_0(B_R\cap B)$, we obtain 
\begin{equation*}
  \int\limits_{B_R\cap B}\abs{\lap{(\ue-v_k)}}^2dx \,\leq\, \Le\left(1-\frac{1}{\mu^2}\right) + \frac{\Le}{\mu^2}\,\int\limits_{B_R\cap B}\abs{\grad\ue}^2dx. 
\end{equation*}
Note that  there holds
\[
 1 < \mu \leq \left(1+\frac{\abs{B_R}}{\abs{\O(\ue)}}\right)^\frac{1}{n}.
\]
and, using Taylor's expansion, we  find
\begin{equation*}
 1-\mu^{-2} \leq  C(n,\abs{\O(\ue)})\,R^n.
\end{equation*}
Thus, 
\begin{align*}
  \int\limits_{B_R(x_0)\cap B}\abs{\lap{(\ue-v_k)}}^2dx \,
  &\leq \, C(n,\Le,\abs{\O(\ue)})\left(R^n+\int\limits_{B_R\cap B}\abs{\grad\ue}^2dx\right).
  \end{align*}
\end{proof}

Note that if we restrict ourselves to the dimensions $n=2$ and $n=3$, we could improve the statement of Lemma \ref{la:est2} using Sobolev's embedding theorem.  This is how the $C^{1,\alpha}$ regularity of $\uei{0}$ is proven in \cite{Sto2015} (cf. Lemma 2.2 and Theorem 2.4 in \cite{Sto2015}). 
Since we now consider any $n\geq 2$, we need the bootstrapping. 
The next lemma is the essential tool for this new approach to the $C^{1,\alpha}$ regularity of $\ue$. It is based on ideas of \cite[Chapter 3]{HanLin}.

\begin{lemma}\label{la:bootstrap1_buck}
 Suppose that for each $0\leq r\leq R_{\eps,k}$ there holds 
 \[
  \int\limits_{B_r(x_0)} \abs{D^2\ue}^2dx \leq M\,r^\mu,
 \]
where  $M>0$ and $\mu \in [0,n)$. 
Then there exists a constant $C(n,\abs{\O(\ue)})>0$ such that for each $0\leq r\leq R_{\eps,k}$
\[
 \int\limits_{B_r(x_0)}\abs{\grad \ue}^2dx \leq C(n,\abs{\O(\ue)})(1+M)\,r^\lambda,
\]
where $\lambda = \mu +2$ if $\mu < n-2$ and $\lambda$ is arbitrary in $(0,n)$ if $n-2\leq \mu < n$.
\end{lemma}
\begin{proof}
 Let $0\leq r\leq s \leq R_{\eps,k}$. For a function $w \in H^{2,2}(B)$ we set 
 \[
  (w)_{r,x_0} := \fint\limits_{B_r(x_0)}w\,dx = \frac{1}{\abs{B_r(x_0)}}\int\limits_{B_r(x_0)}w\,dx.
 \]
Using this notation we write
\[
\int\limits_{B_r(x_0)}\abs{\grad \ue}^2dx = \sum_{i=1}^n \int\limits_{B_r(x_0)}\abs{\partial_i\ue - (\partial_i\ue)_{s,x_0} + (\partial_i\ue)_{s,x_0}}^2dx.
\]
Then Young's inequality implies
\begin{align*}
 \int\limits_{B_r(x_0)}\abs{\grad \ue}^2dx &\leq 2 \sum_{i=1}^n \left( \int\limits_{B_r(x_0)}(\partial_i\ue)^2_{s,x_0}dx + \int\limits_{B_r(x_0)}\abs{\partial_i\ue - (\partial_i\ue)_{s,x_0}}^2dx \right) \\
 &\leq2 \sum_{i=1}^n \left(\abs{B_r} \left(\fint\limits_{B_s(x_0)}\partial_i\ue\,dx \right)^2 + \int\limits_{B_s(x_0)}\abs{\partial_i\ue - (\partial_i\ue)_{s,x_0}}^2dx \right).
\end{align*}

Applying H\"older's and a local version of Poincaré's inequality, we find that
\begin{align*}
 \int\limits_{B_r(x_0)}\abs{\grad \ue}^2dx &\leq C(n) \left[ \left(\frac{r}{s}\right)^n \int\limits_{B_s(x_0)}\abs{\grad \ue}^2dx + s^2\int\limits_{B_s(x_0)}\abs{D^2\ue}^2dx  \right],
\end{align*}
where the constant $C$ only depends on $n$.
By assumption, we can proceed to
\begin{align*}
 \int\limits_{B_r(x_0)}\abs{\grad \ue}^2dx &\leq C(n) \left[\left(\frac{r}{s}\right)^n \int\limits_{B_s(x_0)}\abs{\grad \ue}^2dx +M\, s^{\mu+2} \right].
\end{align*}
Now Lemma \ref{la:tech_morrey} implies that for each $0\leq r\leq s\leq R_\eps$ there holds
\begin{align*}
 \int\limits_{B_r(x_0)}\abs{\grad \ue}^2dx &\leq C(n)  \left(\frac{r}{s}\right)^\lambda \left[\int\limits_{B_{s}(x_0)}\abs{\grad \ue}^2dx + M\,s^\lambda  \right].
 \intertext{where $\lambda = \mu +2$ if $\mu< n-2$ and $\lambda$ is arbitrary in $(0,n)$ if $n-2 \leq \mu < n$. Now we choose $s=R_\eps$. Recalling that $R_\eps$  depends on $\abs{\O(\ue)}$ we deduce}
 \int\limits_{B_r(x_0)}\abs{\grad \ue}^2dx &\leq C(n,\abs{\O(\ue)})(1+M)r^\lambda.
\end{align*}
\end{proof}

Now we are ready to prove the main theorem of this section. 

\begin{theorem}\label{theo:hoelder_reg}
 For every $\eps>0$ and $k\in\{0,1\}$ every minimizer $\ue$ of $\Iei{k}$ is in $C^{1,\alpha}(\overline{B})$ for every $\alpha \in (0,1)$.
\end{theorem}
\begin{proof}
Our aim is to show that for each $x_0\in\overline{B}$ and every $0\leq r \leq R_{\eps,k}$ there holds 
 \begin{equation}\label{eq:assum_morrey}
  \int\limits_{B_r(x_0)\cap B}\abs{D^2\ue}^2 dx \leq C(n,\Le,\abs{\O(\ue)})\,r^{n-2+2\alpha}
 \end{equation}
for each $\alpha \in (0,1)$. 
Then Theorem \ref{theo:morrey} finishes the proof if we choose $\phi = \partial_i \ue$ for $i \in \{1,\ldots,n\}$.
We prove \eqref{eq:assum_morrey} by using a bootstrap argument. 
 Let $x_0 \in \overline{B}$ and $0\leq r\leq R\leq R_{\eps,k}$. 
 Then there obviously holds
 \[
   \int\limits_{B_r(x_0)\cap B}\abs{D^2\ue}^2dx \leq 2 \int\limits_{B_r(x_0)}\abs{D^2v_k}^2dx + 2 \int\limits_{B_R(x_0)}\abs{D^2(\ue-v_k)}^2dx,
 \]
where $v_k$ is defined in \eqref{eq:hat_v}.
 Due to Lemma \ref{biharm_v} and Lemma \ref{la:est2} we obtain
 \begin{equation}\label{eq:boot0}
 \begin{split}
   \int\limits_{B_r(x_0)\cap B}\abs{D^2\ue}^2dx 
  \leq & \, C \left(\frac{r}{R}\right)^n\int\limits_{B_R(x_0)}\abs{D^2\ue}^2dx \\&+ C(n,\Le,\abs{\O(\ue)})\left(R^n + \int\limits_{B_R(x_0)}\abs{\grad\ue}^2dx\right).
\end{split} 
\end{equation}
Now  we start the bootstrapping. 
For every $0\leq r\leq R_{\eps,k}$ there holds
\begin{equation}\label{eq:boot1}
  \int\limits_{B_r(x_0)\cap B}\abs{D^2\ue}^2dx \leq \int\limits_B\abs{D^2\ue}^2dx = \Le = \Le\,r^0. 
\end{equation}
Applying Lemma \ref{la:bootstrap1_buck},  estimate \eqref{eq:boot1} yields for every $0\leq r \leq R_{\eps,k}$
\begin{equation*}
 \int\limits_{B_r(x_0)}\abs{\grad\ue}^2dx \leq C(n,\Le,\abs{\O(\ue)})\,r^{\lambda_0},
\end{equation*}
where $\lambda_0 \in (0,n)$ if $n=2$ and $\lambda_0 = 2$ if $n\geq 3$. 
We insert this estimate in \eqref{eq:boot0}. This yields
\begin{align*}
 \int\limits_{B_r(x_0)\cap B}\abs{D^2\ue}^2dx
 &\leq C \left(\frac{r}{R}\right)^n\int\limits_{B_R(x_0)\cap B}\abs{D^2\ue}^2dx + C(n,\Le,\abs{\O(\ue)}))R^{\lambda_0}
\end{align*}
for every $0\leq r \leq R\leq R_{\eps,k}$. Applying Lemma \ref{la:tech_morrey}, we obtain
\begin{align*}
 \int\limits_{B_r(x_0)\cap B}\abs{D^2\ue}^2dx &\leq C \left(\frac{r}{R}\right)^{\lambda_0}\left(\int\limits_{B_R(x_0)\cap B}\abs{D^2\ue}^2dx + C(n,\Le,\abs{\O(\ue)})R^{\lambda_0}\right) 
\end{align*}
for $0\leq r\leq R\leq R_{\eps,k}$. Then choosing $R =R_{\eps,k}$ gives us
\begin{equation}\label{eq:boot2}
  \int\limits_{B_r(x_0)\cap B}\abs{D^2\ue}^2dx \leq  C(n,\Le,\abs{\O(\ue)})\,r^{\lambda_0}
\end{equation}
for every $0\leq r\leq R_{\eps,k}$. 
If $n=2$, this is \eqref{eq:assum_morrey}. If $n\geq 3$, \eqref{eq:boot2} is an improvement of estimate \eqref{eq:boot1}. In this case, we again apply Lemma \ref{la:bootstrap1_buck} and obtain
\[
   \int\limits_{B_r(x_0)}\abs{\grad\ue}^2dx \leq C(n,\Le,\abs{\O(\ue)})\,r^{\lambda_1},
\]
where $\lambda_1 \in (0,n)$ if $n\in\{3,4\}$ and $\lambda_1 = 4$ if $n>4$.  Together with estimate \eqref{eq:boot0} we find that
\begin{equation*}
  \int\limits_{B_r(x_0)\cap B} \abs{D^2\ue}^2dx \leq C \left(\frac{r}{R}\right)^n \int\limits_{B_R(x_0)\cap B} \abs{D^2\ue}^2dx + C(n,\Le,\abs{\O(\ue)})R^{\lambda_1}
\end{equation*}
for every $0\leq r\leq R\leq R_{\eps,k}$. Then Lemma \ref{la:tech_morrey} implies
\[
  \int\limits_{B_r(x_0)\cap B} \abs{D^2\ue}^2dx \leq C \left(\frac{r}{R}\right)^{\lambda_1}\left(\int\limits_{B_R(x_0)\cap B}\abs{D^2\ue}^2dx + C(n,\Le,\abs{\O(\ue)})R^{\lambda_1}\right) 
\]
for every $0\leq r\leq R\leq R_{\eps,k}$. Choosing $R=R_{\eps,k}$ we gain
\begin{equation}\label{eq:boot3}
  \int\limits_{B_r(x_0)\cap B} \abs{D^2\ue}^2dx \leq C(n,\Le,\abs{\O(\ue)})\,r^{\lambda_1}.
\end{equation}
For $n\in\{3,4\}$, estimate \eqref{eq:boot3} proves the claim. For $n\geq 5$, we repeat the argumentation. In view of \eqref{eq:boot3}, Lemma \ref{la:bootstrap1_buck}  implies 
\[
 \int\limits_{B_r(x_0)}\abs{\grad\ue}^2dx \leq C(n,\Le,\abs{\O(\ue)})\,r^{\lambda_2},
\]
where $\lambda_2= 6 $ if $n> 6$ and $\lambda_2$ is arbitrary in $(0,n)$ if $n\in\{5,6\}$. Again, we insert this estimate in \eqref{eq:boot0} and deduce an improvement of \eqref{eq:boot2}. Repeating this process proves the claim after finite many steps for every $n\geq2$.
\end{proof}

The continuity of $\ue$ implies that $\O(\ue)$ is an open set. Then classical variational arguments show that $\ue$ solves 
\[
 \bilap \ue + \Le\,\lap\ue =0 \; \mbox{ in } \O(\ue).
\]
Moreover, the $C^{1,\alpha}$ regularity of $\ue$ allows us to split $\partial\O(\ue)$ in the following two parts
\[
 \Gamma_{\eps,k}:= \{ x\in\partial\O(\ue): \abs{\grad\ue(x)}=0 \} \;\mbox{ and }\, \Sigma_{\eps,k} := \{ x\in\partial\O(\ue): \abs{\grad\ue(x)}>0 \}.
\]
Then $\Sigma_{\eps,k}$ is part of a nodal line of $\ue$ and, consequently,  $\mathcal{L}^n(\Sigma_{\eps,k})=0$ for all $\eps>0$.
We define 
\[
 \Omega(\ue) := \O(\ue) \cup \Sigma_{\eps,k}
\]
and call $\partial\Omega(\ue)=\Gamma_{\eps,k}$ the free boundary. 
\begin{remark}\label{rem:PDE}
 Note that $\Omega(\ue)$ is an open set in $\R^n$ and $\abs{\Omega(\ue)} = \abs{\O(\ue)}$. Moreover, the minimizer $\ue$ solves 
\[
   \begin{cases} 
 \bilap\ue + \Le \lap\ue =0, &\mbox{ in } \Omega(\ue) \\ 
  \ue = \abs{\nabla \ue} =0, &\mbox{ on } \partial\Omega(\ue).
  \end{cases}
\]
\end{remark}

The following lemma shows that, considering the functional $\Iei{1}$, the set $\Omega(\uei{1})$ is connected. This result is a direct consequence of the strict monotonicity of the penalization term $\pe{1}$. 

\begin{lemma} \label{la:connected_1}
  For every minimizer $\uei{1} \in H^{2,2}_0(B)$ of $\Iei{1}$ the set $\Omega(\uei{1})$ is connected.
\end{lemma}
\begin{proof}
  We prove the claim by contradiction. Without loss of generality we assume that $\Omega(\uei{1})$ consists of two connected components, namely $\Omega_1$ and $\Omega_2$ with $\abs{\Omega_k}\neq 0$ for $k=1,2$. By $u_k$ we denote
\[
   u_k := \begin{cases}
               \uei{1}, & \mbox{ in } \Omega_k \\
		  0,& \mbox{otherwise}
		\end{cases}.
\]
The minimality of $\uei{1}$ for $\Iei{1}$ implies
\begin{align*}
 \Iei{1}(\uei{1}) =   \Lei{1} + \pe{1}(\abs{\Omega_1}+\abs{\Omega_2}) \leq \Ray(u_1) + \pe{1}(\abs{\Omega_1}) = \Iei{1}(u_1).
\end{align*}
Since $\norm{\grad \uei{1}}_{L^2(B)} = 1$, there holds
\begin{align*}
   &\left(\int\limits_{\Omega_1}\abs{\lap u_1} ^2dx + \int\limits_{\Omega_2}\abs{\lap u_2}^2dx\right)\left(1 - \int\limits_{\Omega_2}\abs{\grad u_2}^2dx\right)  \\ \leq \qquad &\int\limits_{\Omega_1}\abs{\lap u_1}^2dx + \left(\pe{1}(\abs{\Omega_1})-\pe{1}(\abs{\Omega_1}+\abs{\Omega_2})\right)\int\limits_{\Omega_1}\abs{\grad u_1}^2dx.
\end{align*}
Now the strict monotonicity of $\pe{1}$ implies
\[
  \int\limits_{\Omega_2}\abs{\lap u_2}^2dx < \Lei{1}\,\int\limits\abs{\grad u_2}^2dx,
\]
which is equivalent to 
\begin{equation}\label{eq:contra_min_connected}
  \Ray(u_2) < \Lei{1}. 
\end{equation}
Comparing $\Iei{1}(\uei{1})$ and $\Iei{1}(u_2)$ we find
\[
   \Iei{1}(\uei{1}) = \Lei{1} + \pe{1}(\abs{\Omega_1}+\abs{\Omega_2}) \leq \Ray(u_2) + \pe{1}(\abs{\Omega_2}) = \Iei{1}(u_2). 
\]
Then \eqref{eq:contra_min_connected} implies
\[
 \Lei{1} + \pe{1}(\abs{\Omega_1}+\abs{\Omega_2}) < \Lei{1} +\pe{1}(\abs{\Omega_2}). 
\]
Obviously, this is contradictory since $\pe{1}$ is strictly increasing. 
\end{proof}
Let us emphasize that for the proof of Lemma \ref{la:connected_1} the actual value of $\abs{\Omega(\uei{1})}$ is irrelevant since $\pe{1}$ is strictly increasing. Considering the functional $\Iei{0}$, and thus the only non-decreasing penalization term  $\pe{0}$, we have to ensure that $\abs{\Omega(\uei{0})}\geq \omega_0$ before we are able to copy the approach of Lemma \ref{la:connected_1} and can deduce that $\Omega(\uei{0})$ is connected. This will be done in Theorem \ref{theo:Vol_0_1} and Lemma \ref{la:connected_0}. 

The next corollary collects direct consequences of Lemma \ref{la:connected_1}
\begin{corollary}\label{cor:optdomain_1} 
  For every minimizer $\uei{1}$ of $\Iei{1}$ the domain $\Omega(\uei{1})$ satisfies $\Lei{1}=\L(\Omega(\uei{1}))$ and $\Omega(\uei{1})$ is an optimal domain for minimizing the buckling eigenvalue among all domains in $B$ with the same measure as $\Omega(\uei{1})$.                                                                                                                                                                                                                                                                       
\end{corollary}
\begin{proof}
 Let $\uei{1} \in H^{2,2}_0(B)$ minimizes the functional $\Iei{\uei{1}}$. The minimality of $\uei{1}$ for $\Iei{1}$ then implies
\begin{align*}
   \Lei{1} + \pe{1}(\abs{\Omega(\uei{1})}) &= \min\{\Iei{1}(v): v\in H^{2,2}_0(B)\} \\ 
 &\leq \min\{\Ray(v) + \pe{1}(\abs{\Omega(v)}): v\in H^{2,2}_0(B), \abs{\Omega(v)}=\abs{\Omega(\uei{2})}\} \\ 
 &\leq \min \{\Ray(v) :v \in H^{2,2}_0(\Omega(\ue)) \} + \pe{1}(\abs{\Omega(\uei{1})}) \\
&= \L(\Omega(\uei{1}))   + \pe{2}(\abs{\Omega(\uei{1})}) \\ 
&\leq \Lei{1} + \pe{1}(\abs{\Omega(\uei{1})}). 
\end{align*}
Thus, the minimizer $\uei{1}$ is a buckling eigenfunction on $\Omega(\uei{1})$. Moreover, there holds 
\[
  \L(\Omega(\uei{1})) = \min\{\Ray(v): v\in H^{2,2}_0(B), \abs{\Omega(v)}=\abs{\Omega(\uei{1})}\}. 
\]
Hence, $\Omega(\uei{1})$ is an optimal domain for minimizing the buckling eigenvalue among all domains in $B$ with the same measure as $\Omega(\uei{1})$.                                                                                                                                                                                                                                                                       
\end{proof}


\section{The volume condition} \label{sec:volcond}
In the following, we analyze the functionals $\Iei{0}$ and $\Iei{1}$ seperately. 

\subsection{The non-rewarding penalization term}\label{sec:nonrewarding}

In this section, we consider the functional $\Iei{0}$. We will show that for every $\eps>0$ and every minimizer $\uei{0}$ of $\Iei{0}$ the  volume of the set $\Omega(\uei{0})$ cannot fall below the value $\omega_0$. This observation allows us to adopt the proof of Lemma \ref{la:connected_1} to show that $\Omega(\uei{0})$ is connected. Thus, analog to Corollary \ref{cor:optdomain_1}, $\Omega(\uei{0})$ minimizes the buckling load among all open subsets of $B$ with the same measure as $\Omega(\uei{0})$.

Recall that the buckling load is decreasing and the penalization term $\pe{0}$ is non-decreasing with respect to set inclusion. 
Since the penalization term grows with slope $\eps^{-1}$ for arguments larger than $\omega_0$, it seems to be natural that the optimal domain $\Omega(\uei{0})$ adjusts itself to the volume $\omega_0$ provided that $\eps$ is chosen small enough. 
Indeed, the fact that $\abs{\Omega(\uei{0})} \leq \omega_0$ for sufficiently small $\eps$ is eventually a consequence of the scaling property of the buckling load, i.\,e. $\L(M) = t^2\L(tM)$ for $t>0$.
Finally, we finish the present section with the proof of Theorem \ref{theo:main}.

The following theorem shows that the volume of the set $\Omega(\uei{0})$ cannot fall below $\omega_0$.

\begin{theorem}\label{theo:Vol_0_1}
  For every $\eps >0$ there holds $\abs{\Omega(\uei{0})} \geq \omega_0$ for every minimizer $\uei{0}$ of $\Iei{0}$.
\end{theorem}
The proof of Theorem \ref{theo:Vol_0_1} follows exactly the idea of the proof of \cite[Theorem 2.2]{Sto2015}. However, since Theorem \ref{theo:Vol_0_1} is crucial for proving our main theorem, Theorem \ref{theo:main}, we present the proof in detail. 
\begin{proof}[Proof of Theorem \ref{theo:Vol_0_1}]
  We prove the claim by contradiction. Let us assume that for an $\eps>0$ there exists a minimizer $\uei{0}$ of $\Iei{0}$ such that $\abs{\Omega(\uei{0})}<\omega_0$. Now choose an $x_0\in\partial\Omega(\uei{0})\setminus \partial B$ and a radius $r>0$ such that 
\begin{equation}\label{eq:vol_0_1_0}
  \abs{B_r(x_0)\cap\{x\in B:\uei{0}(x)\neq 0\}}>0 
\end{equation}
and $\abs{\Omega(\uei{0})\cup B_r(x_0)}\leq \omega_0$. Note that such an $x_0 \in \partial\Omega(\uei{0})\setminus\partial B$ exists since we assume $\abs{\Omega(\uei{0})}<\omega_0$. 
Let $v \in H^{2,2}_0(B)$ with $v-\uei{0} \in H^{2,2}_0(B_r(x_0))$ and 
\[
   \bilap v + \Lei{0} \lap v = 0 \mbox{ in } B_r(x_0). 
\]
We set
\[ \hat{v} := 
  \begin{cases}
     \uei{0}, &\mbox{ in } B\setminus B_r(x_0)\\
     v, &\mbox{ in } B_r(x_0)
  \end{cases}
\]
and compare the $\Iei{0}$-energies of $\uei{0}$ and $\hat{v}$. This leads to the following local inequality:
\begin{equation}\label{eq:vol_0_1_1}
   \int\limits_{B_r(x_0)}\abs{\lap \uei{0}}^2-\abs{\lap v}^2dx \leq \Lei{0} \int\limits_{B_r(x_0)}\abs{\grad \uei{0}}^2-\abs{\grad v}^2dx.
\end{equation}
Using integration by parts and the definition of $v$, we obtain 
\[
   \int\limits_{B_r(x_0)}\abs{\lap \uei{0}}^2-\abs{\lap v}^2dx = \int\limits_{B_r(x_0)}\abs{\lap(\uei{0}-v)}^2 + 2\Lei{0}\int\limits_{B_r(x_0)}\abs{\grad( \uei{0}- v)}^2dx.
\]
Thus, \eqref{eq:vol_0_1_1} becomes 
\[
   \int\limits_{B_r(x_0)}\abs{\lap \uei{0}}^2-\abs{\lap v}^2dx  \leq \Lei{0}\int\limits_{B_r(x_0)}\abs{\grad( \uei{0}- v)}^2dx
\]
and Poincaré's inequality yields 
\begin{equation}\label{eq:vol_0_1_2}
  \int\limits_{B_r(x_0)}\abs{\lap \uei{0}}^2-\abs{\lap v}^2dx  \leq \Lei{0} r^2 \int\limits_{B_r(x_0)}\abs{\lap( \uei{0}- v)}^2dx.
\end{equation}
Provided that the integral in \eqref{eq:vol_0_1_2} does not vanish, \eqref{eq:vol_0_1_2} is contradictory for sufficently small $r$. 

If the integral in \eqref{eq:vol_0_1_2} vanishes, there holds $\uei{0}\equiv v$ in $B_r(x_0)$. Consequently, $\uei{0}$ is analytic in $B_\frac{r}{2}(x_0)$ since $v$ is there analytic as a solution of an ellipitc equation. However, then $\uei{0}$ vanishes in $B_\frac{r}{2}(x_0)$ because of \eqref{eq:vol_0_1_0}. This is contradictory since $x_0\in\partial\Omega(\uei{0})$. This proves the claim.
\end{proof}

Now we are able to adopt the proof of Lemma \ref{la:connected_1} to show that $\Omega(\uei{0})$ is connected.

\begin{lemma}\label{la:connected_0}
  For every $\eps>0$ and every minimizer $\uei{0}$ of $\Iei{0}$ the set $\Omega(\uei{0})$ is connected. 
\end{lemma}
\begin{proof}
 We follow the idea of the proof of Lemma \ref{la:connected_1} and assume that $\Omega(\uei{0})=\Omega_1 \dot{\cup} \Omega_2$ with $\abs{\Omega_k}>0$ for $k = 1,2$.  Let us first consider that $\abs{\Omega(\uei{0})}>\omega_0$. Then there holds
 \begin{equation}\label{eq:conn_1}
   \pe{0}(\abs{\Omega(\uei{0})}) > \pe{0}(\Omega_k) 
 \end{equation}
 for $k= 1,2$. Hence, arguing as in the proof of Lemma \ref{la:connected_1} and replacing the strict monotonicity of $\pe{1}$ by \eqref{eq:conn_1}, we arrive at a contradiction. 
 
 Now let us assume that $\abs{\Omega(\uei{0})}=\omega_0$. Then 
  \begin{equation}\label{eq:conn_2}
   \pe{0}(\abs{\Omega(\uei{0})}) = \pe{0}(\Omega_k) =0
 \end{equation}
 and arguing as in the proof of Lemma \ref{la:connected_1} we obtain 
 \[
   \Iei{0}(u_2)=\Ray(u_2) \leq \Lei{0} = \Iei{0}(\uei{0}).
 \]
 This implies that $u_2$ minimizes the fucntional $\Iei{0}$. Since $\Omega(u_2)=\Omega_2$ and $\abs{\Omega_2}<\omega_0$ this is contradictory to Theorem \ref{theo:Vol_0_1}. Thus, the claim is proven.
\end{proof}

As a consequence of Lemma \ref{la:connected_0}, we get the analog to Corollary \ref{cor:optdomain_1}. 
\begin{corollary}\label{cor:optdomain_0}
For every minimizer $\uei{0}$ of $\Iei{0}$ the set $\Omega(\uei{0})$ is a domain and $\Lei{0} = \L(\Omega(\uei{0}))$. In addition, $\Omega(\uei{0})$ minimzes the buckling load among all open subsets of $B$ with the same measure as $\Omega(\uei{0})$.
\end{corollary}   

The following remark will be helpful to show that for an appropriate choice of $\eps$ the $n$-dimensional Lebesgue measure of $\Omega(\uei{0})$ cannot become larger than $\omega_0$. 

\begin{remark}\label{rem:buckling_load_omega_0}
 Note that the reference domain $B$ compactly contains a ball $B_R(x_0)$ with $\abs{B_R}=\omega_0$. Let $\phi\in H^{2,2}_0(B_R(x_0))$ denote the buckling eigenfunction on $B_R(x_0)$, i.e.
 \[
  \L(B_R(x_0)) = \min\{\Ray(v,B_R(x_0)): v\in H^{2,2}_0(B_R(x_0))\} = \Ray(\phi,B_R(x_0)).
 \]
 Consequently, for every $\eps>0$ there holds
\[
  \Iei{0}(\phi) = \Iei{1}(\phi) =\Ray(\phi,B_R(x_0)) = \L(B_R(x_0)) = \left(\frac{\omega_n}{\omega_0}\right)^\frac{2}{n}\L(B_1).
\]
\end{remark}

\begin{theorem}\label{theo:Vol_0_2}
  Let $\uei{0}$ be a minimizer for $\Iei{0}$. Then there exists a number $\eps_1 = \eps_1(n,\omega_0)$ such that for $0<\eps\leq \eps_1$ there holds 
  \[
    \abs{\Omega(\uei{0})} = \omega_0. 
  \] 
\end{theorem}
\begin{proof}
We claim that the statement of the theorem holds true for 
\begin{equation}\label{eq:eps_1}
  \eps_1 = \eps_1(n,\omega_0) = \left( \frac{\omega_0}{\omega_n}\right)^\frac{2}{n}\,\frac{\omega_0}{\L(B_1)}. 
\end{equation}
We prove by contradiction. Thus, we choose $\eps\leq\eps_1$ and denote by $\uei{0}$ a minimizer of $\Iei{0}$. Assume that there exists a number $\alpha>1$
 such that 
\begin{equation}\label{eq:assum_vol_larger}
  \abs{\Omega(\uei{0})} = \alpha \omega_0.
\end{equation}
Our aim is to contradict \eqref{eq:assum_vol_larger}. 

Since $\Omega(\uei{0}) \subset B$ and $\alpha >1$, the scaled domain $\Omega' := \alpha^{-\frac{1}{n}}\Omega(\uei{0})$ is also contained in $B$ and satisfies $\abs{\Omega'}=\omega_0$. Let $\psi \in H^{2,2}_0(B)$ denote the first buckling eigenfunction on $\Omega'$. Then the minimality of $\uei{0}$ for $\Iei{0}$ implies
\begin{equation}\label{eq:vol_larger_0}
   \L(\Omega(\uei{0})) + \frac{\omega_0}{\eps}(\alpha -1) = \Iei{0}(\uei{0}) \leq \Iei{0}(\psi) = \L(\Omega').
\end{equation}
By scaling we have 
\[
  \L(\Omega') = \L(\alpha^{-\frac{1}{n}}\Omega(\uei{0})) = \alpha^\frac{2}{n}\L(\Omega(\uei{0})). 
\] 
From \eqref{eq:vol_larger_0} we then get
\begin{equation}\label{eq:vol_larger_1}
   \frac{\omega_0}{\eps}(\alpha -1) \leq \left(\alpha^\frac{2}{n}-1\right)\L(\Omega(\uei{0})). 
\end{equation}
Now let $\phi \in H^{2,2}_0(B)$ be as in Remark \ref{rem:buckling_load_omega_0}. Then, due to the assumption \eqref{eq:assum_vol_larger}, there holds
\[
   \L(\Omega(\uei{0})) < \Iei{0}(\uei{0}) \leq \Iei{0}(\phi) = \left(\frac{\omega_n}{\omega_0}\right)^\frac{2}{n}\L(B_1)
\]
and estimate \eqref{eq:vol_larger_1} becomes
\[
   \left(\frac{\omega_0}{\omega_n}\right)^\frac{2}{n}\,\frac{\omega_0}{\L(B_1)}\,\frac{\alpha -1}{\alpha^\frac{2}{n}-1}< \eps. 
\]
Hence, 
\[
  \frac{\alpha-1}{\alpha^\frac{2}{n}-1} < \frac{\eps}{\omega_0}\,\left(\frac{\omega_n}{\omega_0}\right)^\frac{2}{n}\L(B_1). 
\]
With \eqref{eq:eps_1} this implies
\[
1 \leq  \frac{\alpha-1}{\alpha^\frac{2}{n}-1} < \eps\,\eps_1^{-1}
\]
 since $\alpha >1$ and $n\geq 2$. Thus, for any $\eps\leq\eps_1$ we get a contradiction and the assumption \eqref{eq:assum_vol_larger} cannot hold true for any $0<\alpha<1$.  Hence, $\abs{\Omega(\uei{0})}\geq \omega_0$ if $\eps\leq\eps_1$ .Together with Theorem \ref{theo:Vol_0_1}, this proves the claim. 
\end{proof}

Finally, the proof of the main theorem, Theorem \ref{theo:main}, is a direct consequence of the previous results in this section. 

\begin{proof}[Proof of Theorem \ref{theo:main}]
   Let us choose $\eps\leq\eps_1$, where $\eps_1$ is given in Theorem \ref{theo:Vol_0_2} and let $\uei{0}$ be a minimizer of $\Iei{0}$. Then Theorem \ref{theo:Vol_0_2} implies that $ \abs{\Omega( \uei{0})}=\omega_0$ and, due to Corollary \ref{cor:optdomain_0}, there holds 
   \[
      \Iei{0}(\uei{0})=\L(\Omega(\uei{0})). 
   \]
   Now choose an open set $D \subset B$ with $\abs{D} \leq \omega_0$ and denote by $u_D \in H^{2,2}_0(D)$ the buckling eigenfunction on $D$. Then the minimality of $\uei{0}$ for $\Iei{0}$ implies 
   \[
     \L(\Omega(\uei{0})) = \Iei{0}(\uei{0}) \leq \Iei{0}(u_D) = \L(D). 
   \]
   Hence,
   \begin{equation}\label{eq:opt1}
      \L(\Omega(\uei{0})) = \min\{\L(D) : D\subset B, D \mbox{ open, } \abs{D} \leq \omega_0\}.
   \end{equation}
   In addition, $\Omega(\uei{0})$ is connected (see Lemma \ref{la:connected_0}). This proves Theorem \ref{theo:main}.
\end{proof}

Since the existence of an optimal domain among all open subsets of $B$ of given volume is now proven, the next reasonable step would be a qualitative analysis of the free boundary $\partial\Omega(\uei{0})$. 
Following \cite{AltCaf81}, our next aims would by establishing a nondegeneracy result for $\uei{0}$. Considering second order problems (e.g. \cite{AguAltCaf1986,AltCaf81,BaWa09}), these nondegeneracy results are achieved by applying comparision principles, which which are not available for fourth order operators in general. One possible way out of this difficulty is to replace the penalization term $\pe{0}$ by the rewarding penalization term $\pe{1}$. This will be discussed in the next section. 

\subsection{The rewarding penalization term}\label{sec:rewarding}

In this section, we consider the functional $\Iei{1}$. Analog to Theorem \ref{theo:Vol_0_2} we will find that $\abs{\Omega(\uei{1})}$ cannot become larger than $\omega_0$ provided that $\eps\leq\eps_1$. 

It remains to exclude that $\Omega(\uei{1})<\omega_0$. This will be more involved since adopting a scaling argument like the one we used in the proof of Theorem \ref{theo:Vol_0_2} collapses if we cannot guarantee that the scaled version of $\Omega(\uei{1})$ is still contained in the reference domain $B$. 

Choosing the parameter $\eps$ sufficiently small, we will see that one of the following two situations occurs: either $\abs{\Omega(\uei{1})}=\omega_0$ or $\abs{\Omega(\uei{1})}<\omega_0$ and the rescaled domain $\Omega'$ with $\abs{\Omega'}=\omega_0$ cannot be translated into the reference domain $B$. 

In the first case, $\uei{1}$ is a minimizer of the functional $\Iei{0}$ and $\Omega(\uei{1})$ is an optimal domain for minimizing the buckling load in the sense of Theorem \ref{theo:main}. Thus, in this case, we can treat the functionals $\Iei{0}$ and $\Iei{1}$ as equivalent. 

In the second case, we may think of the domain $\Omega(\uei{1})$ as of a domain with thin tentacles, which may all touch the boundary of the reference domain $B$. 
These tentacles eludes the scaling. Consequently, in this case a more local analysis of $\partial\Omega(\uei{1})$ is needed. 
Exemplary, we will see that assuming that $\partial\Omega(\uei{1})$ satisfies a doubling condition, the domain $\Omega(\uei{1})$ fulfills $\abs{\Omega(\uei{1})}=\omega_0$.

We begin this section with the analog result to Theorem \ref{theo:Vol_0_2}.

\begin{theorem}\label{theo:Vol_1_1}
  For $\eps\leq \eps_1$ every minimizer $\uei{1}$ of $\Iei{1}$ satisfies $\abs{\Omega(\uei{1})}\leq \omega_0$.  Thereby, $\eps_1$ is the number given in Theorem \ref{theo:Vol_0_2}. 
\end{theorem}
\begin{proof}
  Let us assume that $\eps\leq \eps_1$ and that $\abs{\Omega(\uei{1})}=\alpha \omega_0$ for an $\alpha >1$. Then arguing in exactly the same way as in the proof of Theorem \ref{theo:Vol_0_2} leads to a contradiction.  
\end{proof}

We would like to repeat the idea of the proof of Theorem \ref{theo:Vol_0_2} assuming that $\abs{\Omega(\uei{1})}=\alpha\omega_0$ for some $\alpha\in(0,1)$ and then define the scaled domain $\Omega' = \alpha^{-\frac{1}{n}}\Omega(\uei{1})$. However, it is not clear if the enlarged domain $\Omega'$ is still contained in $B$. A partial result can be obtained with the help of an argument of M.\,S.\,Ashbaugh and R.\,S.\,Laugesen. 

\begin{remark}\label{rem:AshLau}
  In \cite{AshLau1996}, M.\,S.\,Ashbaugh and R.\,S.\,Laugesen showed that there exists a constant $c_n \in (0,1)$, only depending on the dimension $n$, such that for every domain $\Omega \subset \R^n$ there holds 
\[
  \L(\Omega) > c_n \, \L(\Omega^\#),
\]
where $\Omega^\#$ denotes a ball in $\R^n$ with the same volume as $\Omega$. In addition, $c_n$ tends to $1$ as $n$ tends to infinity. 
\end{remark}

\begin{theorem}\label{theo:VolCond_low}
   Let $\uei{1}$ be a minimizer of $\Iei{1}$ for $\eps\leq\eps_1$. Then there exists a number $\alpha_0 = \alpha_0(n,\eps_1,\eps)$ such that 
   \[
     \abs{\Omega(\uei{1})} \geq \alpha_0 \omega_0. 
   \] 
   Moreover, we have the explicit representation
   \[
      \alpha_0 = \frac{1+\eps\eps_1-\sqrt{1+2\eps\eps_1+(\eps\eps_1)^2-4c_n\eps\eps_1}}{2\eps\eps_1}, 
   \]
   where $\eps_1$ is given by Theorem \ref{theo:Vol_0_2} and $c_n$ is given by Remark \ref{rem:AshLau}.
\end{theorem}
Note that 
\[
  \lim_{n\to\infty}\alpha_0 = 1 \quad \mbox{and } \lim_{\eps\to0}\alpha_0 = c_n.
\]
Together with Theorem \ref{theo:Vol_1_1} this shows that for $\eps\leq\eps_1$ the domain $\Omega(\uei{1})$ satisfies the volume condition asymptotically as the dimension $n$ approaches infinity. 

\begin{proof}[Proof of Theorem \ref{theo:VolCond_low}]
 Let $\uei{1}$ be a minimizer of $\Iei{1}$ and let $\alpha \in (0,1)$ be such that 
 \[
    \abs{\Omega(\uei{1})} = \alpha \omega_0. 
 \]
 Choose $\phi$ as in Remark \ref{rem:buckling_load_omega_0}. Then the minimality of $\uei{1}$ for $\Iei{1}$ implies
 \begin{equation}\label{eq:low_1}
   \L(\Omega(\uei{1})) - \eps\omega_0 (1-\alpha) = \Iei{1}(\uei{1}) \leq \Iei{1}(\phi) = \left(\frac{\omega_n}{\omega_0}\right)^\frac{2}{n}\L(B_1). 
 \end{equation}
 By $\Omega(\uei{1})^\#$ we denote the ball centered in the origin, having the same volume as $\Omega(\uei{1})$. Then 
 \[
   \L(\Omega(\uei{1})^\#) = \left( \frac{\omega_n}{\alpha\omega_0}\right)^\frac{2}{n}\L(B_1)
 \]
 and applying Remark \ref{rem:AshLau} we obtain from \eqref{eq:low_1}
 \begin{align*}
 &c_n\,\L(\Omega(\uei{1})^\#)- \left(\frac{\omega_n}{\omega_0}\right)^\frac{2}{n}\L(B_1) < \eps\,\omega_0\,(1-\alpha) \\
   \Rightarrow \quad &\left( \frac{\omega_n}{\omega_0}\right)^\frac{2}{n}\L(B_1) \left(c_n\alpha^{-\frac{2}{n}}-1\right) < \eps\,\omega_0\,(1-\alpha) \\
   \Rightarrow \quad &\frac{c_n\alpha^{-\frac{2}{n}}-1}{1-\alpha} < \eps\,\left(\frac{\omega_0}{\omega_n}\right)^\frac{2}{n}\frac{\omega_0}{\L(B_1)} = \eps\,\eps_1.
 \end{align*}
 Since $\alpha<1$, there holds $\alpha^{-\frac{2}{n}}\geq \alpha^{-1}$ and we have 
 \begin{equation}\label{eq:low_2}
   \frac{c_n\alpha^{-1}-1}{1-\alpha} \leq \eps\,\eps_1. 
 \end{equation}
 We set $f(\alpha) :=  \frac{c_n\alpha^{-1}-1}{1-\alpha}$. Then $f: (0,1)\to\R$ is smooth, strictly decreasing and 
 \[
   \lim_{\alpha \to 1}f(\alpha) = -\infty \; \mbox{and } \lim_{\alpha\to0}f(\alpha)=\infty. 
 \]
 By $\alpha_0$ we denote the (unique) solution in $(0,1)$ for equality in \eqref{eq:low_2}, i.e. 
 \[
      \alpha_0 = \frac{1+\eps\eps_1-\sqrt{1+2\eps\eps_1+(\eps\eps_1)^2-4c_n\eps\eps_1}}{2\eps\eps_1}.
   \]
Consequently, the strict monotonicity of $f$ implies that \eqref{eq:low_2} can only hold true for $\alpha \in [\alpha_0,1)$. This proves the theorem. 
\end{proof}

From now on, we always consider $0<\eps\leq\eps_1$. Consequently, there holds 
\[
   \alpha_0\omega_0 \leq \abs{\Omega(\uei{1})} \leq\omega_0.
\] 

\begin{theorem}\label{theo:vol_dicho}
  There exists a number $\eps_0 = \eps_0(n,\omega_0)$ such that for $\eps\leq \eps_0$ every minimizer $\uei{1}$ of $\Iei{1}$ satisfies either 
 \begin{enumerate}[label=\alph*)]
    \item $\abs{\Omega(\uei{1})} = \omega_0$  \qquad or 
    \item $\abs{\Omega(\uei{1})}< \omega_0$ and the rescaled domain $t\,\Omega(\uei{1})$ with $\abs{t\,\Omega(\uei{1})}=\omega_0$ is not a subset of $B$. In addition, there exists no translation $\Phi : \R^n\to\R^n$ such that $\Phi(t\,\Omega(\uei{1}))$ is contained in $B$.
   \end{enumerate}
\end{theorem}
\begin{proof}
  We claim that the statement of the theorem holds true for 
  \begin{equation}\label{eq:eps_0}
    \eps_0 := \min\left\{\eps_1, c_n\frac{\L(B_1)}{\omega_0}\frac{2}{n}\left(\frac{\omega_n}{\omega_0}\right)^\frac{2}{n}\right\} = \min\left\{ \eps_1, c_n\frac{2}{n}\eps_1^{-1}\right\},
  \end{equation} 
  where $\eps_1$ is given in Theorem \ref{theo:Vol_0_2}  and $c_n$ is given in Remark \ref{rem:AshLau}. 
  We prove by contradiction.   Let $0<\eps\leq\eps_0$ and $\uei{1}$ be a minimizer of $\Iei{1}$.  Recall that Theorem \ref{theo:VolCond_low} implies $\abs{\Omega(\uei{1})} \in [\alpha_0\omega_0,\omega_0]$ since we assume $\eps\leq\eps_1$. Now let us assume that there exists an $\alpha \in [\alpha_0,1)$ such that 
   \begin{equation}\label{eq:dicho_0}
      \abs{\Omega(\uei{1})} = \alpha\omega_0 . 
   \end{equation}
  We split the proof in two steps.   
  
 Step 1. \quad Assume $\eps \leq \eps_0$ and let $\uei{0}$ be a minimizer of $\Iei{0}$. Recall that Theorem \ref{theo:Vol_0_2} implies that $\abs{\Omega(\uei{0})}=\omega_0$. We now show that, for every $t \in (0,1)$, the buckling eigenfunction $u_{\eps}^t$ of the scaled domain $t\,\Omega(\uei{0})$ does not minimize the functional $\Iei{1}$, i.e. for every $t \in (0,1)$ there holds 
  \begin{equation}\label{eq:dicho_1}
     \Iei{1}(\uei{1}) < \Iei{1}(u_\eps^t).  
  \end{equation}
  For that purpose, let us assume that \eqref{eq:dicho_1} does not hold true. Then there exists a $t\in(0,1)$ such that 
  \[
    \Iei{1}(\uei{1}) = \Iei{1}(u_\eps^t)
  \] 
  and, since $\abs{t\,\Omega(\uei{0})}=t^n\omega_0 < \omega_0$, we obtain
  \begin{equation}\label{eq:dicho_2}
     \Iei{1}(u_\eps^t) = \L(t\,\Omega(\uei{0})) - \eps\omega_0(1-t^n) \leq \L(\Omega(\uei{0})) = \Iei{1}(\uei{0}).
  \end{equation}
  By scaling there holds 
  \[
    \L(t\,\Omega(\uei{0})) = t^{-2} \L(\Omega(\uei{0})).
  \]
  Thus, from \eqref{eq:dicho_2} we get 
  \begin{equation}\label{eq:dicho_2b}
    \L(\Omega(\uei{0})) \left(t^{-2} - 1\right) \leq \eps\omega_0 (1-t^n). 
  \end{equation}
  By $\Omega(\uei{0})^\#$ we denote the ball centered in the origin with the same volume as $\Omega(\uei{0})$. Then 
  \[
    \L(\Omega(\uei{0})^\#) = \left(\frac{\omega_n}{\omega_0}\right)^\frac{2}{n}\L(B_1)
  \]
  and applying Remark \ref{rem:AshLau} we obtain from \eqref{eq:dicho_2b}
  \begin{align*}
     &c_n \L(\Omega(\uei{0})^\#) \left(t^{-2}-1\right) < \eps \omega_0 \left(1-t^n\right)  \\
      \aq  \quad &c_n\,\left(\frac{\omega_n}{\omega_0}\right)^\frac{2}{n}\L(B_1)\left(t^{-2}-1\right) < \eps\,\omega_0\,(1-t^n) \\
   \aq \quad &\frac{t^{-2}-1}{1-t^n} < \eps \, \frac{\omega_0}{\L(B_1)}\,c_n^{-1}\,\left( \frac{\omega_0}{\omega_n}\right)^\frac{2}{n} = \eps\cdot\eps_1\cdot c_n^{-1}.
 \end{align*}
 Since $t <1$ and $\eps\leq\eps_0$ (see \eqref{eq:eps_0}),  we obtain
 \[
   \frac{2}{n} \leq \frac{t^{-2}-1}{1-t^n} < \eps\cdot\eps_1\cdot c_n^{-1} \leq \frac{2}{n}.
 \]
Obviously, this statement is false and we conculde that \eqref{eq:dicho_1} holds true for every $t \in (0,1)$.

Step 2. \quad Let us fix $t_\ast \in (0,1)$ such that $\abs{t_\ast\Omega(\uei{0})} = \abs{\Omega(\uei{1})}$. Hence, according to \eqref{eq:dicho_0}, there holds $t_\ast = \alpha^\frac{1}{n}$ and applying \eqref{eq:dicho_1} we obtain
\[
    \Iei{1}(\uei{1}) < \Iei{1}(u_\eps^{t_\ast}).
\]
By choice of $t_\ast$, this is equivalent to 
\[
  \L(\Omega(\uei{1})) < \L(t_\ast\Omega(\uei{0})) = t_\ast^{-2} \L(\Omega(\uei{0})),
\]
where we additionally used the scaling property of the buckling load. 
Multiplying the above inequality with $t_\ast^2$ and again applying the scaling property of the buckling load, we deduce
\begin{equation}\label{eq:dicho_3}
   \L(t_\ast^{-1}\Omega(\uei{1})) = t_\ast^{-2}\L(\Omega(\uei{1})) < \L(\Omega(\uei{0})).
\end{equation}
Note that $\abs{t_\ast^{-1}\Omega(\uei{1})}=\omega_0$. Thus, if $t_\ast^{-1}\Omega(\uei{1}) \subset B$, \eqref{eq:dicho_3} is contradictory to the minimality of $\Omega(\uei{0})$ for the buckling load among all open subsets of $B$ with volume smaller or equal than $\omega_0$ (see \eqref{eq:opt1}). Consequently, if $t_\ast^{-1}\Omega(\uei{1}) \subset B$, the assumption \eqref{eq:dicho_0} is false and there holds $\abs{\Omega(\uei{1})}=\omega_0$. This proves part a) of the claim. 

If $t_\ast^{-1}\Omega(\uei{1}) \not\subset B$, but there exists a translation $\Phi: \R^n\to\R^n$ such that $\Phi(t_\ast^{-1}\Omega(\uei{1})) \subset B$, we arrive at the same contradiction to \eqref{eq:opt1} as above because of the translational invariance of the buckling load (i.e. $\L(D) = \L(\Phi(D))$ for every  translation $\Phi$). 

Hence, if $\abs{\Omega(\uei{1})}<\omega_0$, the scaled domain $t_\ast^{-1}\Omega(\uei{1})$ cannot be translated into the ball $B$. This proves part b) of the claim.
\end{proof}

 From now on, we always choose $\eps\leq\eps_0$. Let $\uei{1}\in H^{2,2}_0(B)$ be a minimizer of $\Iei{1}$ and let us assume that case a) of Theorem \ref{theo:vol_dicho} holds true. Hence, $\abs{\Omega(\uei{1})}=\omega_0$ and repeating the proof of Theorem \ref{theo:main} we  gave in Section \ref{sec:nonrewarding} we find that $\Omega(\uei{1})$ minimizes the buckling load among all open subsets of $B$ with volume smaller or equal than $\omega_0$. In addition, $\uei{1}$ minimizes the functional $\Iei{0}$.
 
 Consequently, if we could exclude that the case b) of Theorem \ref{theo:vol_dicho} may occur, we could treat the functionals $\Iei{0}$ and $\Iei{1}$ as equivalent.  We will discuss this issue in the following section.
  

\subsection{Discussion of case b) of Theorem \ref{theo:vol_dicho}}\label{sec:case_b}

In this section, we always consider a minimizer $\uei{1} \in H^{2,2}_0(B)$ such that case b) of Theorem \ref{theo:vol_dicho} occurs. Thus, there holds $\abs{\Omega(\uei{1})}<\omega_0$ and the rescaled domain $\Omega'$ with $\abs{\Omega'}=\omega_0$ is neither contained in $B$, nor can it be translated into $B$. Consequently, we have to think of the domain $\Omega(\uei{1})$ as of a domain with tentacle-like, very long and very thin parts. 

Since we do not posses any information about regularity properties of $\partial\Omega(\uei{1})$, we currently cannot anaylze these tentacle-like domains in more detail. 

However, assuming that $\partial\Omega(\uei{1})$ satisfies a doubling condition, we will be able to establish a nondegeneracy result for $\uei{1}$. This nondegeneracy enables us to prove that, potentially after translation, the up the volume $\omega_0$ rescaled version of $\Omega(\uei{1})$ is always contained in $B$. Hence, if $\partial\Omega(\uei{1})$ satisfies a doubling condition, the case b) of Theorem \ref{theo:vol_dicho} cannot occur and there holds $\abs{\Omega(\uei{1})}=\omega_0$

From now on, we always assume that  $\partial\Omega(\uei{1})$ satisfies the following doubling property. Assume that there exists a constant $\sigma>0$ and a radius $0<R_0<1$ such that  for every $x_0 \in \partial\Omega(\uei{1})$ and every $0<R\leq R_0$ there holds 
\begin{equation}\label{eq:doubprop}
  \abs{B_{2R}(x_0)\cap\Omega(\uei{1})} \leq \sigma \, \abs{B_R(x_0) \cap \Omega(\uei{1})}.
\end{equation}
Note carefully that the condition \eqref{eq:doubprop} does not exclude that $\Omega(\uei{1})$ forms thin tentacles but it determines the minimal rate at which the volume of a possible tentacle may decrease.

\subsubsection{Nondegeneracy of $\uei{1}$}
Our aim is to establish a nondegeneracy result for $\uei{1}$.
For convenience, we cite a technical lemma which will be applied in the proof of Lemma~\ref{la:nondeg}. 

\begin{lemma}[{\cite[Lemma 6.1]{Giusti_direct_methods}}]\label{la:giusti}
Let $Z(t)$ be a bounded nonnegative function in the interval $[\rho,R]$. Assume that for $\rho\leq t < s \leq R$ we have
\[
 Z(t) \leq \left[A(s-t)^{-\alpha} + B(s-t)^{-\beta}+C\right] + \vartheta Z(s)
\]
with $A,B,C \geq 0$, $\alpha>\beta>0$ and $0\leq\vartheta<1$. Then,
\[
  Z(\rho) \leq c(\alpha,\theta)\left[ A(R-\rho)^{-\alpha} + B(R-\rho)^{-\beta}+C \right].
\]
\end{lemma}

\begin{lemma}\label{la:nondeg}
 Let $\uei{1}$ be a minimizer of $\Iei{1}$ and  let $\partial\Omega(\uei{1})$ satisfy \eqref{eq:doubprop}.
  There exists a $c_1 = c_1(\eps,n,\omega_0,\sigma) >0$  such that there holds 
  \begin{equation}\label{eq:nondeg_claim}
     c_1 \,R \leq \sup_{B_R(x_0)}\abs{\grad\uei{1}}, 
  \end{equation}
  where $x_0 \in \partial\Omega(\uei{1})$ and $0<R\leq R_0$.
  In particular, $c_1$ is independent of the choice of $x_0$ and $R$. 
\end{lemma} 
\begin{proof}
  Let $x_0 \in \partial\Omega(\uei{1})$ and $0<R\leq R_0$. 
  For $\frac{R}{2} \leq t < s \leq R$, let  $\eta \in C^\infty(\R^n)$ satisfy $0 \leq \eta \leq 1$, $\eta \equiv 0$ in $B_t(x_0)$ and $\eta \equiv 1$ in $R^n\setminus B_s(x_0) $.  Note that for every $x \in B_s(x_0)\setminus B_t(x_0)$ there holds 
  \begin{equation}\label{eq:eta}
     \abs{\eta(x)} \leq \frac{C(n)}{s-t} \; \mbox{and } \abs{\lap\eta(x)}\leq \frac{C(n)}{(s-t)^2}, 
  \end{equation}
  where $C(n)$ only depends on $n$.  We use $\uei{1}\eta$ as  a comparison function for $\Iei{1}$. By the minimality of $\uei{1}$ we obtain
  \[
    \Iei{1}(\uei{1}) \leq \Iei{1}(\uei{1}\eta). 
  \]
   Since, by construction, there holds  $\abs{\O(\eta\uei{1})}= \abs{\O(\uei{1})} - \abs{B_t(x_0)\cap \Omega(\uei{1})}$, straight forward computation yields
  \begin{equation}\label{eq:inner_cond_1}\begin{split}
    \int\limits_{B_s(x_0)}\abs{\lap \uei{1}}^2dx + \eps\,\abs{B_t&(x_0)\cap\Omega(\uei{1})} \leq \int\limits_{B_s\setminus B_t(x_0)}\abs{\lap(\eta\uei{1})}^2dx \\ &+ (\L(\Omega(\uei{1})+\eps\abs{B_t(x_0)\cap\Omega(\uei{1})})\int\limits_{B_s(x_0)}\abs{\grad \uei{1}}^2dx. 
  \end{split}\end{equation}
  A detailed analysis of this inequality gives the claim. This is done in three steps. 
  
  Step 1. \quad We estimate the last summand on the right hand side of \eqref{eq:inner_cond_1}. Recall that 
  \begin{align*}
     \L(\Omega(\uei{1})) &\leq \left(\frac{\omega_n}{\omega_0}\right)^\frac{2}{n}\L(B_1) + \eps\,(\omega_0 - \abs{\O(\uei{1})})\\
     &\leq \left(\frac{\omega_n}{\omega_0}\right)^\frac{2}{n}\L(B_1) + \eps_1\,\omega_0 = C(n,\omega_0). 
  \end{align*}
  Consequently, 
  \[
    \L(\Omega(\uei{1})) + \eps\abs{B_t(x_0)\cap\Omega(\uei{1})} \leq C(n,\omega_0) + \eps_1\abs{B_{R_0}(x_0)} \leq C(n,\omega_0)
  \]
  and \eqref{eq:inner_cond_1} becomes
  \begin{equation}\label{eq:inner_cond_2}
  \begin{split}
       \int\limits_{B_s(x_0)}\abs{\lap \uei{1}}^2dx + \eps\,\abs{B_t(x_0)&\cap\Omega(\uei{1})} \\ &\leq \int\limits_{B_s\setminus B_t(x_0)}\abs{\lap(\eta\uei{1})}^2dx + C(n,\omega_0)\int\limits_{B_s(x_0)}\abs{\grad \uei{1}}^2dx. 
  \end{split}
  \end{equation}
  
  Step 2. \quad Applying Young's inequality we estimate 
  \[
    \abs{\lap(\uei{1}\eta)}^2 \leq 4\left(\abs{\lap \uei{1}}^2\eta^2 + 2\abs{\grad\uei{1}.\grad\eta}^2 + (\uei{1})^2\abs{\lap\eta}^2\right)
  \]
  and together with \eqref{eq:eta} we deduce in $B_s(x_0)\setminus B_t(x_0)$
  \[
     \abs{\lap(\uei{1}\eta)}^2 \leq 4\left(\abs{\lap\uei{1}}^2\eta^2 + 2\frac{C(n)}{(s-t)^2}\abs{\grad\uei{1}}^2 + \frac{C(n)}{(s-t)^4}(\uei{1})^2\right). 
  \]
  With the splitting 
  \[
     \int\limits_{B_s(x_0)}\abs{\lap\uei{1}}^2dx = \int\limits_{B_s\setminus B_t(x_0) }\abs{\lap\uei{1}}^2  dx + \int\limits_{B_t(x_0)}\abs{\lap\uei{1}}^2 dx
  \]
  \eqref{eq:inner_cond_2} becomes
  \begin{align*}
     &\int\limits_{B_t(x_0)}\abs{\lap\uei{1}}^2dx + \eps\,\abs{B_t(x_0)\cap\Omega(\uei{1})} \leq \int\limits_{B_s\setminus B_t(x_0)}\abs{\lap\uei{1}}^2(4\eta^2-1)dx \\
     &+ \int\limits_{B_s\setminus B_t(x_0)}\frac{C(n)}{(s-t)^2}\abs{\grad\uei{1}}^2 + \frac{C(n)}{(s-t)^4}(\uei{1})^2dx + C(n,\omega_0)\int\limits_{B_s(x_0)}\abs{\grad \uei{1}}^2dx,
  \end{align*}
  where $C(n)$ collects all constants only depending on $n$. Since $0\leq\eta\leq 1$, we obtain
\begin{align*}
     &\int\limits_{B_t(x_0)}\abs{\lap\uei{1}}^2dx + \eps\,\abs{B_t(x_0)\cap\Omega(\uei{1})} \leq 3 \int\limits_{B_s\setminus B_t(x_0)}\abs{\lap\uei{1}}^2dx \\
     &\quad+C(n,\omega_0)\, \int\limits_{B_R(x_0)}\left(1+\frac{1}{(s-t)^2}\right)\abs{\grad\uei{1}}^2 + \frac{1}{(s-t)^4}(\uei{1})^2dx .
\end{align*}
Now we add $3\int_{B_t(x_0)}\abs{\lap\uei{1}}^2dx$ to both sides of the above inequality and divide the resulting inequality by $4$. Subsequently, we add $\frac{3\,\eps}{16}\abs{B_s(x_0)\cap\Omega(\uei{1})}$ to the right hand side. This leads to
\begin{equation}\begin{split}\label{eq:inner_cond_3}
    &\int\limits_{B_t(x_0)}\abs{\lap\uei{1}}^2dx + \frac{\eps}{4}\,\abs{B_t(x_0)\cap\Omega(\uei{1})} \leq \frac{3}{4}\left( \int\limits_{B_s(x_0)}\abs{\lap\uei{1}}^2dx + \frac{\eps}{4}\abs{B_s(x_0)\cap\Omega(\uei{1})}\right) 
    \\& \quad+C(n,\omega_0)\, \int\limits_{B_R(x_0)}\left(1+\frac{1}{(s-t)^2}\right)\abs{\grad\uei{1}}^2 + \frac{1}{(s-t)^4}(\uei{1})^2dx .
\end{split}\end{equation}
Setting 
\[
  Z(t) := \int\limits_{B_t(x_0)}\abs{\lap\uei{1}}^2dx + \frac{\eps}{4}\,\abs{B_t(x_0)\cap\Omega(\uei{1})},
\]
estimate \eqref{eq:inner_cond_3} enables us to apply Lemma \ref{la:giusti} and  we obtain
\[
     \int\limits_{B_\frac{R}{2}(x_0)}\abs{\lap\uei{1}}^2dx + \frac{\eps}{4}\,\abs{B_\frac{R}{2}(x_0)\cap\Omega(\uei{1})} \leq
   C(n,\omega_0)\, \int\limits_{B_R(x_0)}\left(1+\frac{1}{R^2}\right)\abs{\grad\uei{1}}^2 + \frac{1}{R^4}(\uei{1})^2dx.
\]
Step 3.\quad The $C^{1,\alpha}$ regularity of $\uei{1}$ allows us to estimate 
\[
   \abs{\uei{1}(x)} \leq 2\,R\,\sup_{B_R(x_0)}\abs{\grad \uei{1}} 
\]
for every $x\in B_R(x_0)$. Moreover, we assume $R\leq R_0<1$. Hence, we find
\begin{equation}\label{eq:why_dp}
\begin{split}
      \int\limits_{B_\frac{R}{2}(x_0)}\abs{\lap\uei{1}}^2dx + \frac{\eps}{4}\,\abs{B_\frac{R}{2}(x_0)&\cap\Omega(\uei{1})} \\ & \leq
   C(n,\omega_0)\, R^{-2}\,\sup_{B_R(x_0)}\abs{\grad \uei{1}}^2\,\abs{B_R(x_0)\cap\Omega(\uei{1})}.
\end{split}
\end{equation}
Since we assume the doubling property \eqref{eq:doubprop} to hold true, omitting the nonnegative integral on the left hand side we obtain
\[
 \frac{\eps}{4} \leq C(n,\omega_0)\,\sigma\, R^{-2}\,\sup_{B_R(x_0)}\abs{\grad \uei{1}}^2.
\]
This proves the claim.
\end{proof}

Note that the assumption \eqref{eq:doubprop} enables us to compare $\abs{B_\frac{R}{2}(x_0)\cap\Omega(\uei{1})}$ with $\abs{B_R(x_0)\cap\Omega(\uei{1})}$ in estimate \eqref{eq:why_dp}. This assumption is only needed because we currently do not have any further information about the free boundary $\partial \Omega(\uei{1})$ and could be replaced by regularity properties of the free boundary. However, let us emphasize that the rewarding property of the penalization term $\pe{1}$ is crucial for proving Lemma \ref{la:nondeg} and cannot be replaced since the rewarding term yields the strictly positive lower bound in \eqref{eq:nondeg_claim}. 

The nondegeneracy of $\uei{1}$ along the free boundary according to Lemma \ref{la:nondeg} allows us to establish a lower bound on the density quotient of $\Omega(\uei{1})$. 

\begin{lemma}\label{la:density}
  Let $\uei{1}$ be a minimizer of $\Iei{1}$, $\alpha \in (0,1)$ and let $\partial\Omega(\uei{1})$ satisfy \eqref{eq:doubprop}. 
  There exists a constant $c_2 = c_2(\eps,n,\omega_0,\sigma,\alpha)>0$ such that for every $x_0 \in \partial\Omega(\uei{1})$ and every $0<R\leq R_0$ there holds 
  \[
     c_2 \,\abs{B_R}^\frac{1-\alpha}{\alpha} \leq \frac{\abs{\Omega(\uei{1})\cap B_R(x_0)}}{\abs{B_R}}.
  \]
\end{lemma} 
Although this lower bound on the density quotient is admittedly weak, it suffices to prove that $\Omega(\uei{1})$ can be rescaled to the volume $\omega_0$ without leaving the reference domain $B$ provided the radius of $B$ is chosen sufficiently large (see Theorem \ref{theo:radius}). 

\begin{proof}[Proof of Lemma \ref{la:density}.] 
  Let $x_0 \in \partial\Omega(\uei{1})$ and $0<R\leq R_0$. According to Lemma \ref{la:nondeg} and  the $C^{1,\alpha}$ regularity of $\uei{1}$ there exists an $x_1 \in \Omega(\uei{1})\cap\overline{B_\frac{R}{2}(x_0)}$ such that 
 \[
    c_1 \frac{R}{2} \leq \sup_{B_\frac{R}{2}(x_0)}\abs{\grad\uei{1}} = \abs{\grad \uei{1}(x_1)}.
 \]
 Now choose $x_2 \in \partial\Omega(\uei{1})$ such that 
 \[
   d:=  \dist(x_1,\partial\Omega(\uei{1}) =  \abs{x_1-x_2}.
 \]
 Since $\abs{\grad \uei{1}(x_2)}=0$, we obtain 
 \[
    c_1 \frac{R}{2} \leq \abs{\grad\uei{1}(x_1)-\grad\uei{1}(x_2)} \leq L_\alpha\,d^\alpha,
 \]
 where $L_\alpha=L(n,\omega_0,\alpha)$ denotes the $\alpha$-Hölder coefficient of $\grad\uei{1}$. By construction, there holds  $B_d(x_1)\subset \Omega(\uei{1})\cap B_R(x_0)$. Consequently, we may proceed to
 \begin{align*}
     \left(\frac{c_1}{2L_\alpha}\right)^n\,R^n \leq d^{\alpha\,n} &\aq \left(\frac{c_1}{2L_\alpha \omega_n^{1-\alpha}}\right)^n\,\abs{B_R} \leq \abs{B_d(x_1)}^\alpha \\
     &\Rightarrow  \left(\frac{c_1}{2L_\alpha \omega_n^{1-\alpha}}\right)^n\,\abs{B_R} \leq \abs{\Omega(\uei{1})\cap B_R(x_0)}^\alpha.
 \end{align*}
 This proves the claim.
\end{proof}

\subsubsection{The volume condition for $\Omega(\uei{1})$}

The next theorem is the key observation to show that $\Omega(\uei{1})$ has the volume $\omega_0$ provided that $\partial\Omega(\uei{1})$ satisfies \eqref{eq:doubprop}.  
It is a consequence of the lower bound on the density quotient according to Lemma \ref{la:density}. 

\begin{theorem}\label{theo:radius}
 Let $\eps\leq\eps_1$ and let $B = B_{R_B}(0)$. Provided that $R_B$ is chosen sufficiently large, for every minimizer $\uei{1}$ of $\Iei{1}$ such that $\partial\Omega(\uei{1})$ satisfies \eqref{eq:doubprop} there  holds 
\begin{enumerate}[label=\alph*)]
  \item $\Omega(\uei{1})$ is compactly contained in $B_{2^{-\frac{1}{n}}R_B}(0)$ \; or
  \item there exists a translation $\Phi: \R^n\to\R^n$ such that $\Phi(\Omega(\uei{1}))$ is compactly contained in $B_{2^{-\frac{1}{n}}R_B}(0)$. 
\end{enumerate}
\end{theorem}
\begin{proof}
Let us think of the reference domain $B$ as of a ball centered at the origin and with radius $R_B$.
In addition, let $\uei{1}$ be a minimizer of $\Iei{1}$ for $\eps\leq\eps_1$. 
In order to prove the claim, let us assume that $\Omega(\uei{1})$ is not compactly contained in $B_{2^{-\frac{1}{n}}R_B}(0)$. For the sake of convenience, we abbreviate $S := 2^{-\frac{1}{n}}R_B$. Hence, we assume $\partial\Omega(\uei{1})\cap\partial B_{S}(0) \neq \emptyset$. 

Of course, there either holds $0 \in \Omega(\uei{1})$ or $0\not\in\Omega(\uei{1})$. At first, we will handle the case where the origin is already contained in $\Omega(\uei{1})$. Secondly, we will show that we may translate $\Omega(\uei{1})$ such that the origin becomes an inner point of $\Omega(\uei{1})$. 

Step 1. \quad We consider that $0\in\Omega(\uei{1})$. 
  
Note that for every $m\in\N$ with $m\geq 3$ there holds 
  \begin{equation}\label{eq:B_covering}
      B_S(0) = \bigcup_{i=0}^{m-2} B_{\frac{i+2}{m}S}(0) \setminus B_{\frac{i}{m}S}(0). 
  \end{equation}
  Since we assume that $\partial\Omega(\uei{1})\cap\partial B_S(0) \neq \emptyset$ and $0\in\Omega(\uei{1})$, there exists a smallest index $i_0=i_0(m)$ such that for each $i\geq i_0$ there exists an $x_i \in \partial\Omega(\uei{1})\cap\partial B_{\frac{i+1}{m}S}(0)$.  
  
 We fix $m\in\N$ such that $\frac{S}{m} \leq R_0.$ In addition, we fix an $\alpha \in (0,1)$.
 Then applying Lemma \ref{la:density}  for  $i_0\leq i\leq m-2$ we obtain
\begin{equation}\label{eq:radius_1}
c_2 \abs{B_\frac{S}{m}}^\frac{1}{\alpha} \leq \abs{\Omega(\uei{1})\cap B_\frac{S}{m}(x_i)}. 
\end{equation}
We now sum \eqref{eq:radius_1} from $i=i_0(m)$ to $i=m-2$. Since $B_\frac{S}{m}(x_i)\cap B_\frac{S}{m}(x_k)=\emptyset$ for $i\neq k$ and $\abs{\Omega(\uei{1})}\leq\omega_0$, this implies
\[
 c_2\,(m-1-i_0(m)) \, \abs{B_{\frac{S}{m}}}^\frac{1}{\alpha} \leq \sum_{i=i_0}^{m-2}\abs{\Omega(\uei{1})\cap B_\frac{S}{m}(x_i)} \leq \abs{\Omega(\uei{1})} \leq \omega_0.
\]
Note that since $B_{\frac{i_0}{m}S}(0)\subset\Omega(\uei{1})$ and $\abs{\Omega(\uei{1})}\leq\omega_0$,  $i_0(m)$ is bounded. 
Indeed, 
\[
\abs{B_{\frac{i_0}{m}S}} \leq \abs{\Omega(\uei{1})} \leq \omega_0 
\]
implies
\[
  i_0(m) \leq \left(\frac{\omega_0}{\omega_n}\right)^\frac{1}{n}\frac{m}{S}.
\]
Specifying the choice of $m\in \N$ such that 
\[
   \frac{R_0}{4} \leq \frac{S}{m}\leq \frac{R_0}{2}
\]
and recalling that $S:=2^{-\frac{1}{n}}R_B $, we obtain
\[
   c_2\,\left(\frac{2\,R_B}{2^\frac{1}{n}R_0}-1-\frac{4}{R_0}\left(\frac{\omega_0}{\omega_n}\right)^\frac{1}{n}\right) \abs{B_\frac{R_0}{4}}^\frac{1}{\alpha} \leq c_2 \,(m-1-i_0(m))\abs{B_\frac{S}{m}}^\frac{1}{\alpha} \leq \omega_0.
\]
Since this estimate is false if $R_B$ is chosen sufficiently large, the proof is finished provided that $0 \in \Omega(\uei{1})$. 
 
 Step 2. \quad Let us now assume that the origin is not contained in $\Omega(\uei{1})$. However, there exists an $x_0 \in \Omega(\uei{1})\cap B$. We now translate $\Omega(\uei{1})$ such that $x_0$ is translated to the origin, i.e. we consider 
 \[
  \Phi : \R^n \to \R^n,\, x \mapsto x-x_0. 
 \]
 We call $\Omega' :=\Phi(\Omega(\uei{1}))$ and $v_\eps(x) := \uei{1}(\Phi^{-1}(x))$. Thus, 
 \[
   \Omega' = \{x\in\R^n : v_\eps(x)\neq 0 \mbox{ or } (v_\eps(x)=0 \wedge \abs{\grad v_\eps(x)}>0)\} 
 \]
 and $v_\eps\in H^{2,2}_0(\Omega')$. Moreover, $\partial\Omega' = \Phi(\partial\Omega(\uei{1}))$. 
 
 Let us emphasize that, in general, $\Omega'$ may not be contained in $B$ and, thus, $v_\eps \not\in H^{2,2}_0(B)$.
 Note carefully that in Step 1 the minimality of $\uei{1}$ for $\Iei{1}$ is not used explicitly.  However, the minimality is necessary to establish Lemma \ref{la:nondeg} and, subsequently, Lemma \ref{la:density} and the application of Lemma \ref{la:density} leads to estimate \eqref{eq:radius_1}, which is the crucial observation in Step 1. 
 
 If $y_0 \in \partial\Omega'$ and $0<R\leq R_0$, then there exists a $z_0 \in \partial\Omega(\uei{1})$ such that $y_0 = \Phi(z_0)$. Lemma \ref{la:density} together with the translational invariance of the Lebesgue measure then imply
\begin{equation}\label{eq:radius_5}
    \abs{\Omega' \cap B_R(y_0)} = \abs{\Omega(\uei{1})\cap B_R(z_0)} \geq c_2 \abs{B_R}^\frac{1}{\alpha}.
\end{equation}
 Estimate \eqref{eq:radius_5} enables us to repeat the approach presented in Step 1.
Again we consider the segmentation \eqref{eq:B_covering} and assume that $\partial\Omega'\cap\partial B_S(0)$ is not empty. Then there exists a smallest index $i_0(m)$ such that for every $i_0\leq i \leq m-2$ there exists an $x_i \in \partial\Omega'\cap\partial B_{\frac{i+1}{m}S}(0)$. Now applying \eqref{eq:radius_5}  we obtain for $i_0\leq i \leq m-2$
\[
c_2\,\abs{B_\frac{S}{m}}^\frac{1}{\alpha} \leq \abs{\Omega' \cap B_\frac{S}{m}(x_i)}
\]
Since $\abs{\Omega'} = \abs{\Omega(\uei{1})} \leq \omega_0$ we may repeat the argumentation from Step 1 and obtain that $\Omega'=\Phi(\Omega(\uei{1}))$ is compactly contained in $B_{2^{-\frac{1}{n}}R_B}$ and, by construction, contains the origin. 
 \end{proof}
 
 As a direct consequence of Theorem \ref{theo:radius} we deduce that, if  $\partial\Omega(\uei{1})$ satisfies the doubling condition \eqref{eq:doubprop}, the domain $\Omega(\uei{1})$ satisfies $\abs{\Omega(\uei{1})}=\omega_0$.
 
 \begin{corollary}
   Let $\eps\leq\eps_0$ and let $\uei{1}\in H^{2,2}_0(B)$ minimize $\Iei{1}$. Provided that the radius $R_B$ of $B$ is chosen sufficiently large and that $\partial\Omega(\uei{1})$ satisfies the doubling property \eqref{eq:doubprop}, there holds 
   $\abs{\Omega(\uei{1})}=\omega_0$.
 \end{corollary}
 \begin{proof}
     Let $\eps\leq\eps_0$ and let  $\uei{1}\in H^{2,2}_0(B)$ be a minimizer of $\Iei{1}$ such that $\partial\Omega(\uei{1})$ satisfies \eqref{eq:doubprop}. In addition, we assume $\abs{\Omega(\uei{1})}=\alpha\omega_0$ for an $\alpha \in [\alpha_0,1)$. Note that since $c_n\geq \frac{1}{2}$ for every $n \in \N$, the quantity $\alpha_0$ given in Theorem \ref{theo:VolCond_low} satisfies
     \begin{equation}\label{eq:bound_alpha_0}
        \alpha_0\geq \frac{1+\eps_1^2-\sqrt{(1+\eps_1^2)^2-4c_n\eps_1^2}}{2\eps_1^2} \geq \frac{1+\eps_1^2-\sqrt{1+\eps_1^4}}{2\eps_1^2} \geq \frac{1}{2} 
     \end{equation}
     for every choice of $\omega_0$, which determines $\eps_1$. Applying Theorem \ref{theo:radius} the domain $\Omega(\uei{1})$ is compactly contained in $B_{2^{-\frac{1}{n}}R_B}$ (possibly after translation) and estimate \eqref{eq:bound_alpha_0} implies that the scaled domain $\alpha^{-\frac{1}{n}}\Omega(\uei{1})$ is a subset of $B_{R_B}$. This is a contradiction to Theorem \ref{theo:vol_dicho} since we assume that $\abs{\Omega(\uei{1})}<\omega_0$.
 \end{proof}

We finish with some concluding remarks.

\begin{remark}
The fundamental tone of a domain $\Omega \subset \R^n$ is defined as 
\[
  \Gamma(\Omega) := \min\left\{ \frac{\int_\Omega\abs{\lap v}^2dx}{\int_\Omega v^2dx}: v \in H^{2,2}_0(\Omega)\right\}. 
\]
Considering the functional $\Ie : H^{2,2}_0(B) \to \R$ by 
\[
  \Iei{k}(v) := \frac{\int_\Omega\abs{\lap v}^2dx}{\int_\Omega v^2dx} +\pe{k}(\abs{\O(v)}),
\]
where $\pe{k}$ and $\O(v)$ are defined as above, the approach presented in the Sections \ref{sec:penprob} and  \ref{sec:volcond} can be adopted to prove the existence of a connected domain $\Omega^*$ with $\abs{\Omega^*}=\omega_0$ such that
\[
  \Gamma(\Omega^*) = \min\{\Gamma(D): D\subset B, D \mbox{open}, \abs{D}\leq\omega_0 \}.
\]
\end{remark}

\begin{remark}
 Future work should address the following issues. 
\begin{itemize}
\item It remains an open problem how to prove regularity of the free boundaries $\partial\Omega(\uei{0})$ and $\partial\Omega(\uei{1})$, respectively. 
\item It remains open to prove the doubling property we assumed in Section \ref{sec:case_b} or to find another argument why the the enlarged domain $\alpha^{-\frac{1}{n}}\Omega(\uei{1})$ is still contained in $B$.
\item It remains open to prove existence of an optimal domain for minimizing the buckling load among all (probably unbounded) open subsets of $\R^n$ with given measure. 
\end{itemize}
\end{remark}

\noindent\textbf{Acknowledgment.} The author is funded by the Deutsche Forschungsgemeinschaft (DFG, German Research Foundation) - project number 396521072.

\end{document}